\documentclass{amsart}

\usepackage{amsmath,amsthm,amssymb}

\theoremstyle{plain}
\newtheorem{rem}{Remark}
\newtheorem{definition}{Definition}
\newtheorem{thm}[definition]{Theorem}

\newtheorem{lem}[definition]{Lemma}

\newtheorem{rei}[definition]{Example}

\newcommand{\Mat}{{\rm Mat}}
\newcommand{\Det}{{\rm Det}}
\newcommand{\sgn}{{\rm sgn}}

\begin{document}
\title[]{Proof of some properties of transfer using noncommutative determinants}
\author[N. Yamaguchi]{Naoya Yamaguchi}
\date{\today}

\maketitle

\begin{abstract}
A transfer is a group homomorphism from a group to an abelian quotient group of a subgroup of finite index. 
In this paper, we give a natural interpretation of the transfers in group theory in terms of noncommutative determinants. 
\end{abstract}

\section{\bf{Introduction}}
A transfer is defined by Issai Schur \cite{Schur} as a group homomorphism from a group to an abelian quotient group of a subgroup of the group. 
In finite group theory, transfers play an important role in transfer theorems. 
Transfer theorems include, for example, Alperin's theorem \cite[Theorem~4.2]{ALPERIN}, 
Burnside's theorem \cite[Hauptsatz~4.2.6]{Huppert}, 
and Hall-Wielandt's theorem \cite[Theorem~14.4.2]{HALL}.

On the other hand, Eduard Study defined the determinant of a quaternionic matrix \cite{Aslaksen}. 
The Study determinant uses a regular representation from $\Mat(n, \mathbb{H})$ to $\Mat(2n, \mathbb{C})$, 
where $\mathbb{H}$ is the quaternions. 
Similarly, we define a noncommutative determinant. 
It is similar to the Dieudonn\'e determinant \cite{Artin}.

T\^oru Umeda suggested that a transfer can be derived as a noncommutative determinant \cite[Footnote 7]{Umeda}. 
In this paper, we develop his ideas in order to explain the properties of the transfers by using noncommutative determinants. 
As a result, we give a natural interpretation of the transfers in group theory in terms of noncommutative determinants.

Let $G$ be a group, 
$H$ a subgroup of $G$ of finite index, 
$K$ a normal subgroup of $H$, 
and the quotient group $H/K$ of $K$ in $H$ an abelian group. 
The transfer of $G$ into $H/K$ is a group homomorphism $V_{G \rightarrow H/K} : G \rightarrow H/K$. 
The definition of the transfer $V_{G \rightarrow H/K}$ uses the left (or right) coset representatives of $H$ in $G$. 
We can show that a transfer has the following properties. 
\begin{enumerate}
\item A transfer is a group homomorphism from $G$ to $H/K$ \cite[Theorem~$3.1$]{Mayer}. 
\item A transfer is invariant under a change of coset representatives \cite[Proposition~$3.1$]{Mayer}. 
\item A transfer by left coset representatives equals a transfer by right coset representatives \cite[Section~$3.1$]{Mayer}. 
\end{enumerate}

Let $R$ be a commutative ring with unity and 
$RG$ the group algebra of $G$ over $R$ whose elements are all possible finite sums of the form $\sum_{g \in G} x_{g} g, x_{g} \in R$. 
The noncommutative determinant uses a left (or right) regular representation from $RG$ to $\Mat(m, RH)$, 
where $m$ is the index of $H$ in $G$. 
Our main result is the following.

\begin{thm}\label{thm:0}
We can regard the transfer $V_{G \rightarrow H/K}$ as the noncommutative determinant $\Det{}$. 
That is, we have 
$$
\Det{(g)} = \sgn(g) V_{G \rightarrow H/K}(g) \quad (g \in G)
$$
where the map $\sgn : G \rightarrow \{-1, 1\}$ is a group homomorphism and $1$ is the unit element of $R$. 
In addition, we can show that the above properties of the transfer (1), (2), and (3) by the following properties of the noncommutative determinant $\Det{}$. 
\begin{enumerate}
\item[$(1')$] The determinant is a multiplicative map from $RG$ to $R(H/K)$. 
\item[$(2')$] The determinant is invariant under a change of a regular representation. 
\item[$(3')$] The determinant of any left regular representation equals the determinant of any right regular representation. 
\end{enumerate}
\end{thm}

\section{\bf{Definition of the transfer}}
Here, we define the left and right transfer of $G$ into $H/K$. 

Let $G = t_{1} H \cup t_{2} H \cup \cdots \cup t_{m} H$. 
That is, we take a complete set $\{ t_{1}, t_{2}, \ldots, t_{m} \}$ of left coset representatives of $H$ in $G$. 
We define $\overline{g} = t_{i}$ for all $g \in t_{i}H$. 
The definition of the left transfer is the following. 

\begin{definition}[Left transfer {\cite[Definition~$3.3$]{Mayer}}]\label{def:2.1.1}
We define the map $V_{G \rightarrow H/K} : G \rightarrow H/K$ by
$$
V_{G \rightarrow H/K}(g) = \prod_{i = 1}^{m} \left\{ \left(\overline{g t_{i}} \right)^{-1} g t_{i} \right\} K. 
$$
We call the map $V_{G \rightarrow H/K}$ the left transfer of $G$ into $H/K$. 
\end{definition} 

Next, we define the right transfer of $G$ into $H/K$. 

Let $G = H u_{1} \cup H u_{2} \cup \cdots \cup H u_{m}$. 
That is, we take a complete set $\{ u_{1}, u_{2}, \ldots, u_{m} \}$ of right coset representatives of $H$ in $G$. 
We define $\widetilde{g} = u_{i}$ for all $g \in Hu_{i}$. 
The definition of the right transfer is the following. 

\begin{definition}[Right transfer {\cite[Definition~$3.3$]{Mayer}}]\label{def:2.1.2}
We define the map $\widetilde{V}_{G \rightarrow H/K} : G \rightarrow H/K$ by
$$
\widetilde{V}_{G \rightarrow H/K}(g) = \prod_{i = 1}^{m} \left\{ u_{i} g \left( \widetilde{u_{i} g} \right)^{-1} \right\} K. 
$$
We call the map $\widetilde{V}_{G \rightarrow H/K}$ the right transfer of $G$ into $H/K$. 
\end{definition} 

The definitions of the left and right transfers use the coset representatives of $H$ in $G$. 
But, we can show that the left and right transfers are invariant under a change of coset representatives. 
Furthermore, we can show that a transfer is a group homomorphism from $G$ to $H/K$ and 
a transfer by left coset representatives equals a transfer by right coset representatives.

\section{\bf{Definition of the noncommutative determinant}}
Here, we define the noncommutative determinant. 

First, we define the left regular representation of $RG$. 
We take a complete set $T = \{ t_{1}, t_{2}, \ldots, t_{m} \}$ of left coset representatives of $H$ in $G$. 
Then, for all $\alpha \in RG$, there exists a unique $L_{T}(\alpha) \in \Mat(m, RH)$ such that 
$$
\alpha (t_{1} \quad t_{2} \quad \cdots \quad t_{m}) = (t_{1} \quad t_{2} \quad \cdots \quad t_{m}) L_{T}(\alpha), 
$$
where we regard $\alpha (t_{1} \cdots \ t_{m})$ as scalar multiplication $(\alpha t_{1} \cdots \alpha t_{m})$. 
The $R$-algebra homomorphism $L_{T} : RG \ni \alpha \mapsto L_{T}(\alpha) \in \Mat(m, RH)$ is called the left regular representation with respect to $T$. 

Let $T' = \{ t'_{1}, t'_{2}, \ldots, t'_{m} \}$ be an another complete set of left coset representatives of $H$ in $G$. 
Then, there exists $P \in \Mat(m, RH)$ such that $L_{T} = P^{-1} L_{T'} P$.

\begin{rei}\label{rei:2.1.3}
Let $G = \mathbb{Z}/ 2 \mathbb{Z} = \{ \overline{0}, \overline{1} \}$, 
$H = \{ \overline{0} \}$, and $\alpha = x \overline{0} + y \overline{1} \in RG$. 
Then, we have 
\begin{align*}
\alpha (\overline{0} \quad \overline{1}) = (\overline{0} \quad \overline{1}) 
\begin{bmatrix}
x \overline{0} & y \overline{0} \\ 
y \overline{0} & x \overline{0} 
\end{bmatrix}. 
\end{align*}
\end{rei}

To obtain an expression for $L_{T}$, we define the indicator function $\dot{\chi}$ by 
\begin{align*}
\dot{\chi}(g) = 
\begin{cases}
1 & g \in H, \\ 
0 & g \notin H 
\end{cases}
\end{align*}
for all $g \in G$.

\begin{lem}\label{lem:2.1.4}
Let $\alpha = \sum_{g \in G} x_{g} g$. 
Then, we have 
$$
L_{T}(\alpha)_{i j} = \sum_{g \in G} \dot{\chi} \left( t_{i}^{-1} g t_{j} \right) x_{g} t_{i}^{-1} g t_{j}. 
$$
\end{lem}
\begin{proof}
We have 
\begin{align*}
&( t_{1} \quad t_{2} \quad \cdots \quad t_{m}) \left( \sum_{g \in G} \dot{\chi} \left( t_{i}^{-1} g t_{j} \right) x_{g} t_{i}^{-1} g t_{j} \right)_{1 \leq i \leq m, 1 \leq j \leq m} \\ 
&= \left( 
\sum_{i = 1}^{m} \sum_{g \in G} \dot{\chi} ( t_{i}^{-1} g t_{1} ) x_{g} g t_{1} \quad 
\cdots \quad 
\sum_{i = 1}^{m} \sum_{g \in G} \dot{\chi} ( t_{i}^{-1} g t_{m} ) x_{g} g t_{m} 
\right) \\ 
&= \left( \sum_{g \in G} x_{g} g \right) ( t_{1} \quad t_{2} \quad \cdots \quad t_{m}). 
\end{align*}
This completes the proof. 
\end{proof}

From Lemma~$\ref{lem:2.1.4}$, we have
\begin{align*}
L_{T}(g)_{i j} &= \dot{\chi} \left( t_{i}^{-1} g t_{j} \right) t_{i}^{-1} g t_{j} \\ 
&= 
\begin{cases}
t_{i}^{-1} g t_{j} & t_{i}^{-1} g t_{j} \in H, \\ 
0 & t_{i}^{-1} g t_{j} \not\in H. 
\end{cases}
\end{align*}
From $t_{i}^{-1} g t_{j} \in H$ if and only if $\overline{g t_{j}} = t_{i}$, we have 
\begin{align*}
L_{T}(g)_{i j} = 
\begin{cases}
\left( \overline{g t_{j}} \right)^{-1} g t_{j} & t_{i}^{-1} g t_{j} \in H, \\ 
0 & t_{i}^{-1} g t_{j} \not\in H. 
\end{cases}
\end{align*}

As for the definition of the noncommutative determinant, 
let $\psi : \Mat(m, RH) \rightarrow \Mat \left( m, R \left( H/K \right) \right)$ be an $R$-linear map such that 
$$
\psi \left( h E_{ij} \right) = (hK) E_{ij}
$$
for all $h \in H$ and $1 \leq i, j \leq m$, where $E_{ij}$ is the matrix with $1$ in the $(i \; j)$ entry and $0$ otherwise. 
Obviously, $\psi$ is an $R$-algebra homomorphism. 
The definition of the noncommutative determinant is the following. 

\begin{definition}\label{def:2.1.5}
We define the map $\Det : RG \rightarrow R \left( H/K \right)$ by 
$$
\Det = \det{} \circ \psi \circ L_{T}. 
$$
\end{definition}

Since there is $P$ such that $L_{T} = P^{-1} L_{T'} P$, we have 
\begin{align*}
\Det 
&= \det{} \circ \psi \circ L_{T} \\ 
&= \det{} \circ \psi \circ L_{T'}. 
\end{align*}
Thus, the determinant is invariant under a change of left regular representations, 
so the determinant $\Det$ is well-defined. 
If $K$ is the commutator subgroup of $H$, 
the determinant is similar to the Dieudonn\'e determinant.

Obviously, the map $\Det$ is a homomorphism. 
That is, $\Det(\alpha \beta) = \Det(\alpha) \Det(\beta)$ for all $\alpha, \beta \in RG$. 
Therefore, we obtain properties~$(1')$ and $(2')$.

\begin{rem}\label{}
In general, that $\alpha \in RG$ is invertible is not equivalent to that $\Det{(\alpha)} \in R(H/K)$ is invertible. 
For example, 
let $R = \mathbb{C}$, 
$\mathbb{Z}/2\mathbb{Z} = \left\{ \overline{0}, \overline{1} \right\}$ be the group of order two, 
$S_{3}$ be the symmetric group of degree three, 
$G = \mathbb{Z}/2\mathbb{Z} \times S_{3}$, 
$H = S_{3}$, 
and $K = \left[H, H \right]$ the commutator subgroup of $H$. 
 Then 
$\alpha = \left(\overline{0}, e \right) + \left(\overline{0}, (1 2 3) \right) + \left( \overline{0}, (1 3 2) \right)$ is not invertible, 
where $e$ is the unit element of $H$. 
But, $\Det(\alpha) = 9 K$ is invertible. 
\end{rem}

\section{\bf{Proof of the properties}}
Here, we prove the transfer properties by using the noncommutative determinant's properties. 

For all $g \in G$ and for all $t \in T$, there exists a unique $t_{j} \in T$ such that $t_{i}^{-1} g t_{j} \in H$. 
Therefore, there exists $\sgn(g) \in \{ -1, 1 \}$ such that 
\begin{align*}
\Det(g) 
&= \det{\left( \psi \left( L_{T}(g) \right) \right)} \\ 
&= \sgn(g) \prod_{i = 1}^{m} \left\{ \left( \overline{g t_{i}} \right)^{-1} g t_{i} \right\} K \\ 
&= \sgn(g) V_{G \rightarrow H/K}(g). 
\end{align*}

Thus, we have 
\begin{align*}
\sgn(g h) V_{G \rightarrow H/K}(g h) 
&= \Det(g h) \\ 
&= \Det(g) \Det(h) \\ 
&= \sgn(g) \sgn(h) V_{G \rightarrow H/K}(g) V_{G \rightarrow H/K}(h). 
\end{align*}

Hence, we obtain 
\begin{align*}
\sgn(g h) &= \sgn(g) \sgn(h), \\ 
V_{G \rightarrow H/K}(g h) &= V_{G \rightarrow H/K}(g) V_{G \rightarrow H/K}(h). 
\end{align*}

Therefore, from property~$(1')$ that $\Det$ is a homomorphism, 
the left transfer $V_{G \rightarrow H/K}$ is a group homomorphism (Assuming, that is, $R = \mathbb{F}_{2}$, and we do not consider the signature).

Next, we show that the left transfer is invariant under a change of coset representatives by using property~$(2')$ that the determinant is invariant under a change of regular representations. 
That is, we show that 
$$
\prod_{i = 1}^{m} \left\{ \left( \overline{g t_{i}} \right)^{-1} g t_{i} \right\} K = \prod_{i = 1}^{m} \left\{ \left( \overline{\overline{g t'_{i}}} \right)^{-1} g t'_{i} \right\} K 
$$
where we define $\overline{\overline{g}} = t'_{i}$ for all $g \in t'_{i} H$.

From property~$(2')$, 
there exists $\sgn'(g) \in \{ -1, 1 \}$ such that 
\begin{align*}
\prod_{i = 1}^{m} \left\{ \left( \overline{g t_{i}} \right)^{-1} g t_{i} \right\} K 
&= \sgn(g) \Det(g) \\ 
&= \sgn(g) \sgn'(g) \prod_{i = 1}^{m} \left\{ \left( \overline{\overline{g t'_{i}}} \right)^{-1} g t'_{i} \right\} K. 
\end{align*}
Therefore, we have $\sgn(g) \sgn'(g) = 1$ and 
$$
\prod_{i = 1}^{m} \left\{ \left( \overline{g t_{i}} \right)^{-1} g t_{i} \right\} K = \prod_{i = 1}^{m} \left\{ \left( \overline{\overline{g t'_{i}}} \right)^{-1} g t'_{i} \right\} K. 
$$
Hence, the left transfer is invariant under a change of coset representatives.

Now let us prove property~$(3)$ that $V_{G \rightarrow H/K} = \widetilde{V}_{G \rightarrow H/K}$ from property~$(3')$ that any left regular representation is equivalent to any right regular representation.

Let $G = Hu_{1} \cup Hu_{2} \cup \cdots \cup Hu_{m}$. 
That is, we take a complete set $U= \{ u_{1}, u_{2}, \ldots, u_{m} \}$ of right coset representatives of $H$ in $G$. 
Then, for all $\alpha \in RG$, there exists $R_{U}(\alpha) \in \Mat(m, RH)$ such that 
\begin{align*}
\begin{pmatrix} 
u_{1} \\ 
u_{2} \\ 
\vdots \\ 
u_{m}
\end{pmatrix} 
\alpha 
= R_{U}(\alpha) 
\begin{pmatrix}
u_{1} \\ 
u_{2} \\ 
\vdots \\ 
u_{m}
\end{pmatrix}. 
\end{align*}
The $R$-algebra homomorphism $R_{U} : RG \ni \alpha \mapsto R_{U}(\alpha) \in \Mat(m, RH)$ is called the right regular representation. 

The same as the left transfer, 
we can show that the following lemma. 

\begin{lem}\label{lem:2.1.7}
Let $\alpha = \sum_{g \in G} x_{g} g$. Then, we have 
$$
R_{U}(\alpha)_{i j} = \sum_{g \in G} \dot{\chi} (u_{i} g u_{j}^{-1}) x_{g} u_{i} g u_{j}^{-1}. 
$$
\end{lem}

Therefore, there exists $\widetilde{\sgn}(g) \in \{ -1, 1 \}$ such that 
$$
\left( \det{} \circ \psi \circ R_{U} \right)(g) = \widetilde{\sgn}(g) \widetilde{V}_{G \rightarrow H/K}(g) 
$$
and 
$\widetilde{V}_{G \rightarrow H/K}$ is invariant under a change of coset representatives of $H$ in $G$. 
We have properties~$(1)$ and $(2)$.

Since $T$ is a complete set of left coset representatives of $H$ in $G$, 
we can take a complete set of $T^{-1} = \{ t_{1}^{-1}, t_{2}^{-1}, \ldots, t_{m}^{-1} \}$ of right coset representatives of $H$ in $G$. 
Therefore, 
\begin{align*}
R_{T^{-1}}(\alpha)_{i j} 
&= \sum_{g \in G} \dot{\chi} \left( t_{i}^{-1} g (t_{j}^{-1})^{-1} \right) x_{g} t_{i}^{-1} g (t_{j}^{-1})^{-1} \\ 
&= L_{T}(\alpha)_{i j}. 
\end{align*}
We obtain property~$(3')$. 
As a result, 
$$
\left( \det{} \circ \psi \circ R_{U} \right)(g) = \left( \det{} \circ \psi \circ L_{T} \right)(g). 
$$
Therefore, we have 
\begin{align*}
\widetilde{\sgn}(g) = \sgn(g), \quad \widetilde{V}_{G \rightarrow H/K} = V_{G \rightarrow H/K}. 
\end{align*}
We obtain property~($3$). 
This completes the proof of Theorem~$\ref{thm:0}$.

\clearpage

\thanks{\bf{Acknowledgments}}
I am deeply grateful to Prof. T\^oru Umeda, Prof. Hiroyuki Ochiai, Prof. Minoru Itoh and Prof. Hideaki Morita who provided helpful comments and suggestions. 
In particular, T\^oru Umeda gave me the initial motivation for undertaking this study. 
I would also like to thank Cid Reyes for comments and suggestions. 
This work was supported by a grant from the Japan Society for the Promotion of Science (JSPS KAKENHI Grant Number 15J06842).

\medskip
\begin{flushleft}
Naoya Yamaguchi\\
Center for Co-Evolutional Social Systems\\
Kyushu University\\
Nishi-ku, Fukuoka 819-0395 \\
Japan\\
n-yamaguchi@math.kyushu-u.ac.jp
\end{flushleft}

\end{document}